\documentclass[12pt]{amsart}
\usepackage{amssymb}
\usepackage{amsfonts}
\usepackage{epsfig}
\usepackage{amsmath}
\usepackage{graphicx}
\usepackage{caption}
\usepackage{subcaption}
\usepackage{wrapfig}
\usepackage
[       left=1.5in,
        right=1in,
        top=1in,
        bottom=1in,
]
{geometry}

\usepackage{mathtools}
\newtheorem{thm}{Theorem}

\newtheorem{lem}[thm]{Lemma}
\newtheorem{coro}[thm]{Corollary}

\newtheorem{conj}{Conjecture}

\theoremstyle{definition}

\newcommand{\wt}{\mathcal{W}_T}
\newcommand{\wrt}{\mathcal{W}_R}

\begin{document}

\title{Bipyramid decompositions of multicrossing link complements}

\begin{abstract}
Generalizing previous constructions, we present a dual pair of decompositions of the complement of a link $L$ into bipyramids, given any multicrossing projection of $L$. When $L$ is hyperbolic, this gives new upper bounds on the volume of $L$ given its multicrossing projection. These bounds are realized by three closely related infinite tiling weaves.
\end{abstract}

\author[Colin Adams]{Colin Adams}
\address{Department of Mathematics and Statistics, Williams College, Williamstown, MA 01267}
\email{Colin.C.Adams@williams.edu}

\author{Gregory Kehne}
\address{Department of Mathematical Sciences, Carnegie Mellon University, Pittsburgh, PA 15213}
\email{gkehne@andrew.cmu.edu}
\date{\today}

\maketitle

\section{Introduction} The standard planar projections of links embedded in $S^3$ are \textit{2-crossing projections}, which means that strands only cross one another two at a time in the projection---or equivalently, that the vertices of the projection graph $G$ all have degree $4$. Recently, 2-crossing projections have been generalized to \textit{$n$-crossing projections}, which are projections of a link in which all strands cross one another $n$ at a time; equivalently, $G$ is $2n$-regular.

Every link has a projection consisting solely of $n$-crossings for every $n \ge 2$, and many of the ideas that apply to 2-crossing projections can be generalized to $n$-crossing projections. See for instance \cite{Adams2012}, \cite{SMALL2014}, and  \cite{Adams2013}. More generally, \textit{multicrossing projections} are projections $P$ in which crossings of varying multiplicities are permitted, and the vertices of $G$ need only be of even degree.

A particular case of interest is the \textit{\"ubercrossing projections}, which are the link projections consisting of a lone multicrossing. When the multiplicity of the lone multicrossing is odd and each strand is connected to both of its neighboring strands in the multicrossing, it is a \textit{petal projection} of a knot. Petal knot projections and \"ubercrossing projections were shown to exist for any knot and link (respectively) in \cite{SMALL2012}, and further studied in \cite{SMALL2013} and \cite{Hass2014}.

When referencing an $n$-crossing $c$ in a projection of a link, it will be useful to refer in a standard way to the strands $s_1,\dots,s_n$ and the levels $l_1,\dots,l_n$ at which each of the strands passes through the $n$-crossing. These labels are shown in Figure \ref{schematic}. Throughout we refer to the levels of adjacent strands $l_i, l_{i+1}$ and the strands corresponding to adjacent levels $s_i, s_{i+1}$; the `wraparound cases' $l_n, l_1$ and $s_n, s_1$ are implied. We keep track of a particular multicrossing by noting the permutation $l_1, \ldots l_n$ obtained by considering the levels of its strands clockwise from above, starting on the top level of the crossing.

\begin{figure}[h]
\centering
    \begin{subfigure}{0.45\textwidth}
    	\centering
        \includegraphics[scale = 0.2]{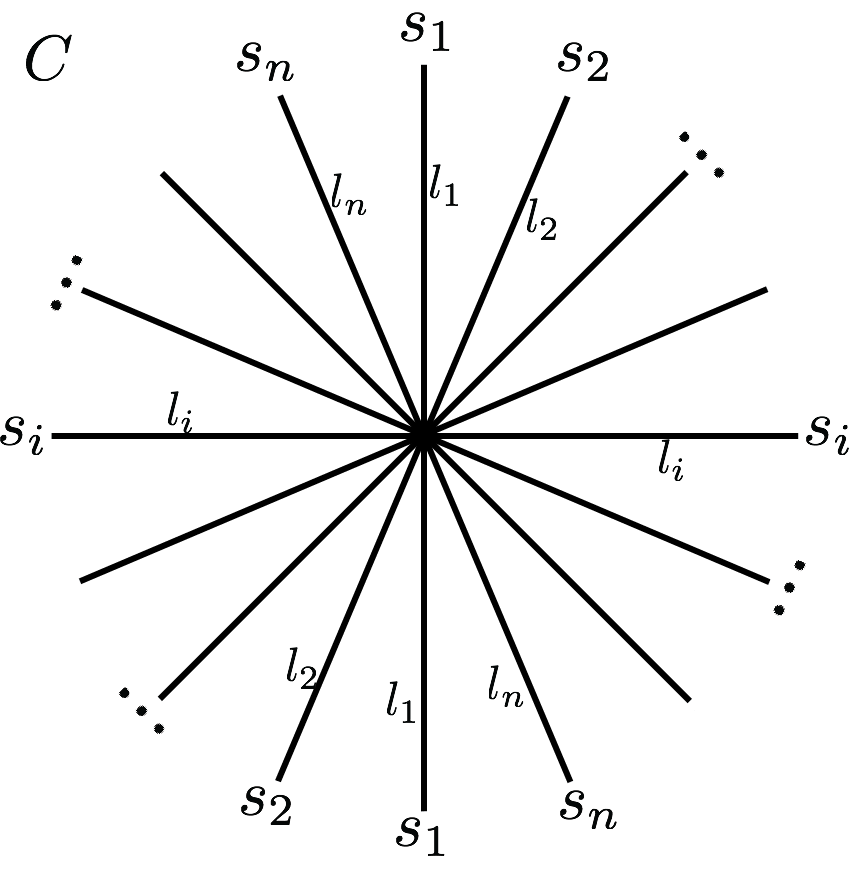}
        \caption{A general multicrossing $c$ of size $n$, with strands and associated strand levels labelled.}
        \label{diagram}
    \end{subfigure}
    \quad
   \quad
    \begin{subfigure}{0.45\textwidth}
    	\centering
        \includegraphics[scale=0.215]{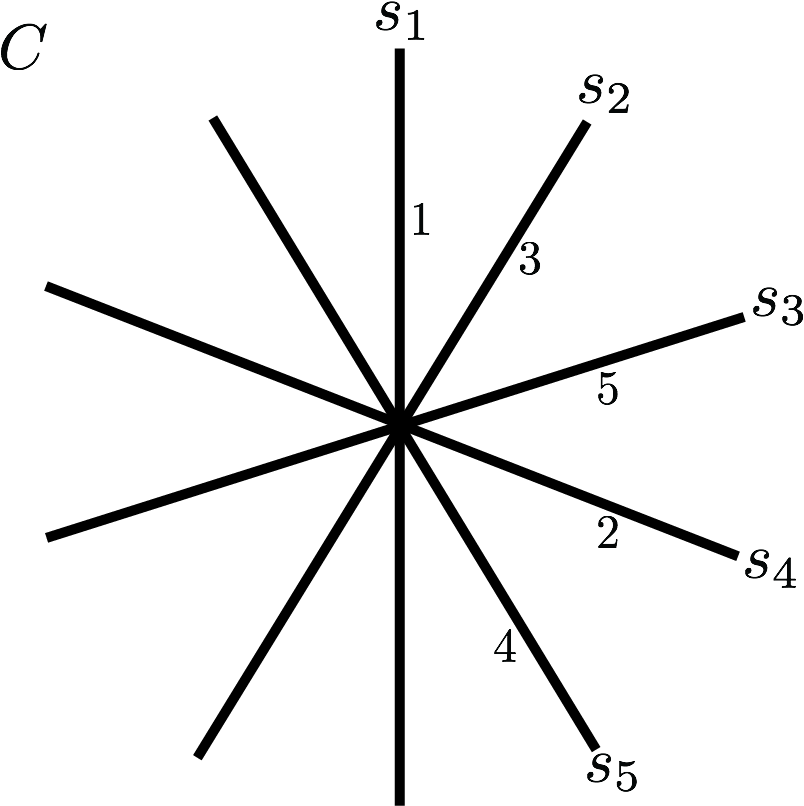}
        \caption{The specific multicrossing 13524, which appears in an \"ubercrossing projection of the figure-eight knot.}
        \label{41diagram}
    \end{subfigure}
\label{schematic}
\caption{multicrossings are represented by their strands and the levels at which the strands enter the crossing. The levels of the strands give a permutation on $n$ elements that encodes the multicrossing.}
\end{figure}

By work of W. Thurston \cite{ThurstonNotes},  the complements of many knots and links admit a hyperbolic metric, and when this is the case, the Mostow/Prasad Rigidity Theorem says the metric is uniquely determined. This implies that geometric invariants of a link complement derived from its  hyperbolic metric are topological invariants of the link. The \textit{hyperbolic volume} of a link, defined to be $vol(L)=vol(S^3\setminus L)$, is particularly discerning invariant for distinguishing among hyperbolic knots and links.  Additionally, it offers an interesting measure of the complexity of a link.

As hyperbolic 3-manifolds, the complements of links produce polyhedral fundamental domains in the universal covering space $\mathbb{H}^3$, where the components of the link correspond to collections of \textit{ideal points} on $\partial\mathbb{H}^3$ ``at infinity''.  The corresponding  decompositions of the manifold $S^3\setminus L$ into hyperbolic polyhedra contain ideal vertices and edges that extend to them. Decompositions can also contain \textit{finite vertices}, which correspond to points in the interior of the link complement. For any such combinatorial polyhedron $P$, there is an upper bound on the hyperbolic volume $vol(P)$ in $\mathbb{H}^3$ which holds across all embeddings of $P$ in $\mathbb{H}^3$ with geodesic faces and edges. We make special use of the maximal hyperbolic volume across all octahedra, $v_{oct}\approx 3.6639$, which is realized by the ideal regular octahedron.

Since each polyhedron has a maximum achievable volume when embedded in $\mathbb{H}^3$,  a decomposition of a link complement into combinatorial polyhedra provides an upper bound on the volume of the link.

D. Thurston showed that given a 2-crossing projection $P$ of a link $L$, the complement $S^3\setminus L$ can be constructed by placing an octahedron at each crossing, where each octahedron's top and bottom vertices are ideal points in the cusps of the upper and lower strands of its crossing, as in Figure \ref{octa}(A) (see \cite{DThurston}). These octahedra are \textit{crossing-centered bipyramids}, which are bipyramids in the complement of a link with finite equatorial vertices and top and bottom vertices that are ideal points at the cusps of adjacent-level strands in a multi-crossing. The equatorial vertices of each octahedron are pulled up and down to finite vertices $U$ and $D$, which sit above and below the projection plane in the complement of the link. The edges extending from the ideal top and bottom vertices of the octahedra become ``vertical'' semi-finite edges from $U$ or $D$ to the cusps after gluing, and the finite equatorial edges of the octahedra become ``vertical'' edges between $U$ and $D$ that pass through each face in the projection. The edge labellings in Figure \ref{octa} depict these operations. The faces of the crossing-centered octahedra glue to the faces of the octahedra at adjacent crossings on the strand, and they together form a decomposition of the link complement into octahedra. This gives an upper bound of
$$vol(L)\leq c(L)v_{oct},$$ where $c(L)$ is the crossing number.

This decomposition and associated bounds have been applied and modified, as in \cite{SMALL2015a}, \cite{Gar} and \cite{Murakami}. These and other bounds are described in \cite{Adams2015}, where the octahedra are each cut into four tetrahedra as in Figure \ref{octa}(B), then recombined about the finite ``vertical'' edges that pass perpendicularly through the faces of $G$ in the projection plane with endpoints $U$ and $D$. The tetrahedra glue together about each of these edges to form a bipyramid that we call a \textit{face-centered bipyramid}, with finite top and bottom vertices at $U$ and $D$ and ideal equatorial vertices corresponding to the cusps.  For a given 2-crossing projection the face-centered bipyramid decomposition offers a more stringent upper bound on the volume of a link than the octahedral decomposition. This is because the volume of a maximum-size hyperbolic bipyramid grows logarithmically in $|B|$, where $|B|$ is the number of equatorial edges of $B$, and is referred to as the {\it size} of $B$. The volume bound derived from this face-centered bipyramid decomposition is referred to as the FCB bound.

%(Picture of crossing-centered octahedral decomp transforming into face-centered bipyramid decomp.)
\begin{figure}
\centering
   \begin{subfigure}{0.5\textwidth}
    	\centering
        \includegraphics[scale = 0.15]{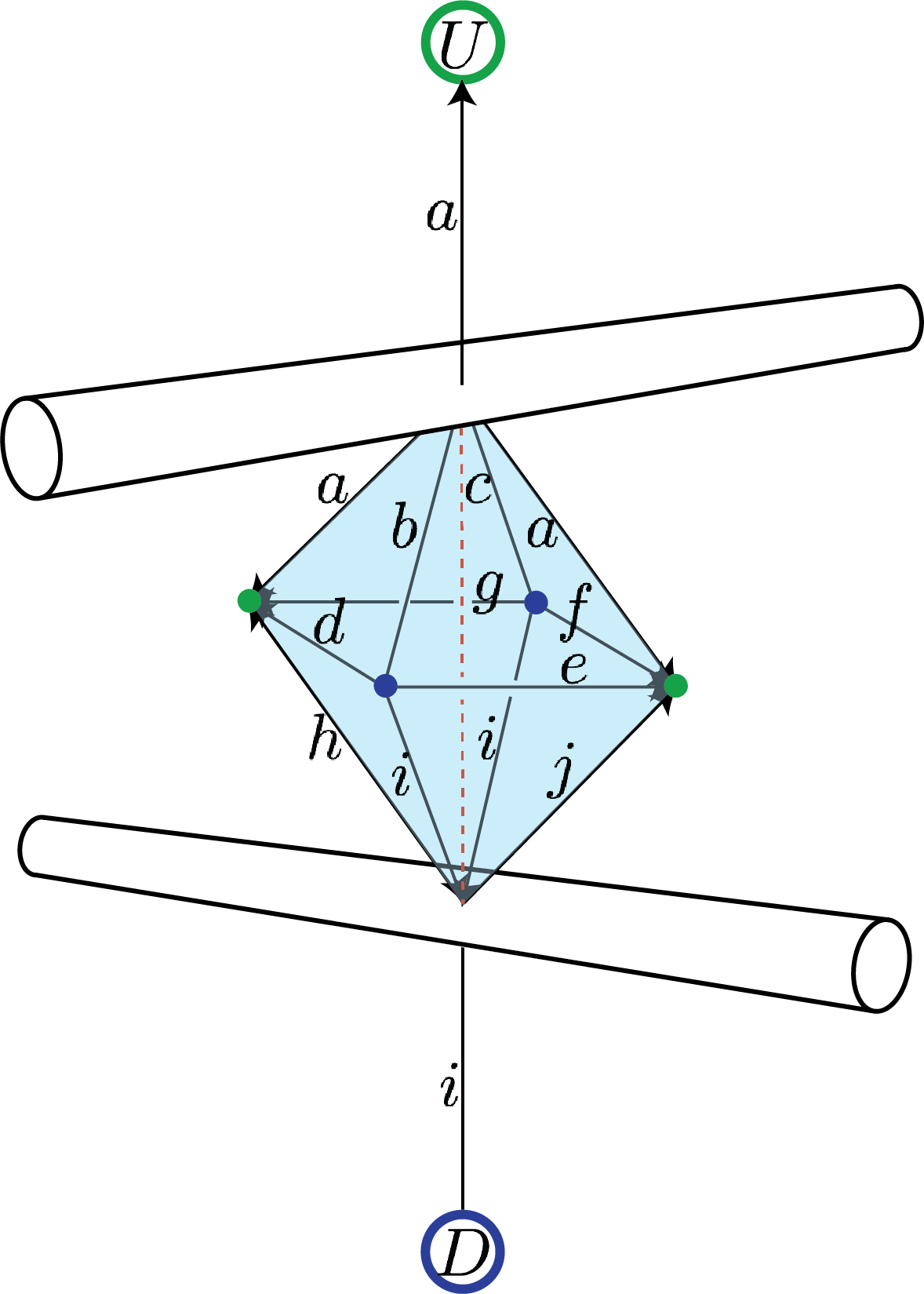}
        \caption{The crossing-centered octahedron.}
     \label{octahedronpic}
    \end{subfigure}
    ~
    \begin{subfigure}{0.5\textwidth}
   	\centering
        \includegraphics[scale=0.15]{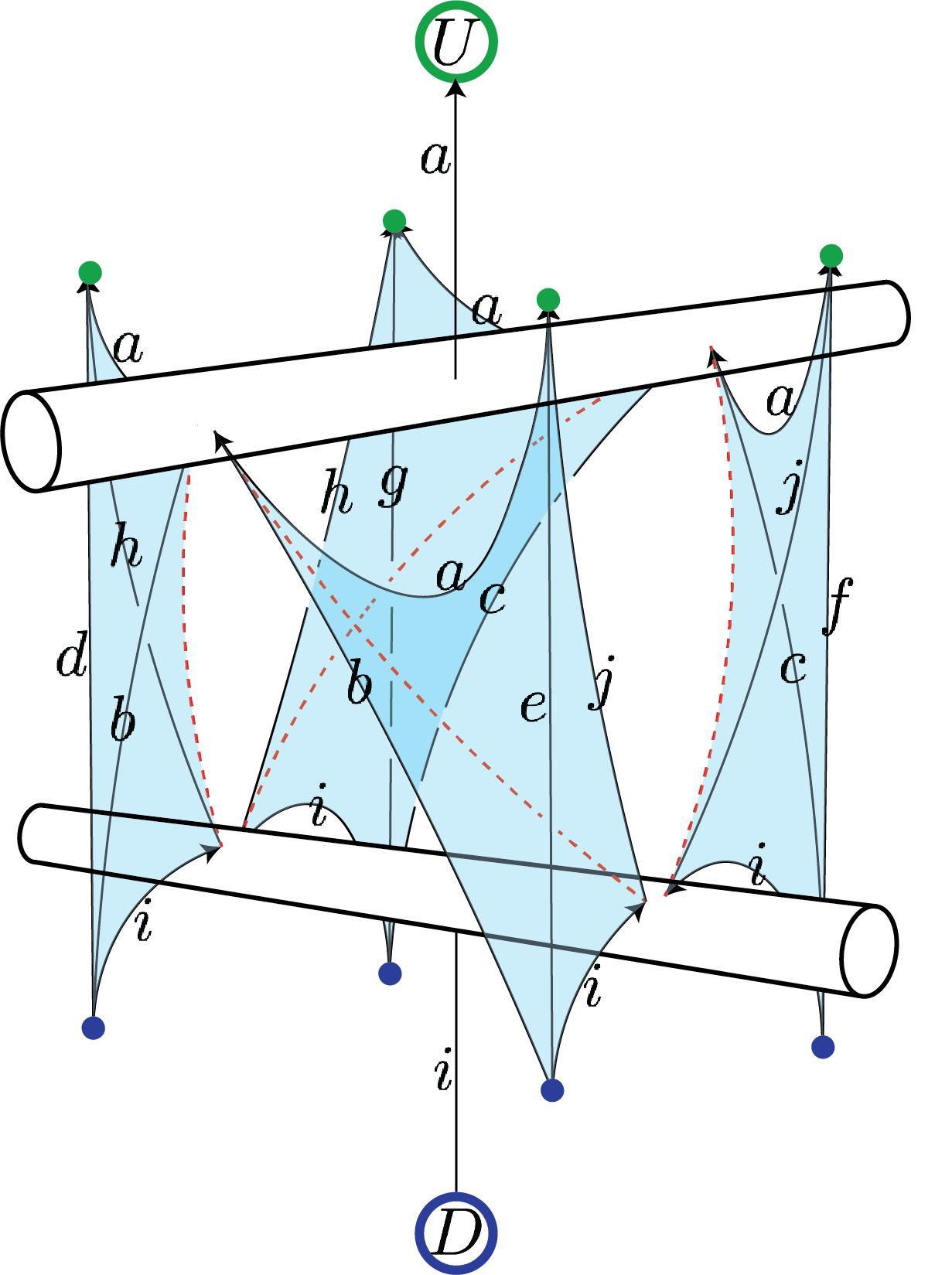}
        \caption{The octahedron is split into tetrahedra.}
       \label{octahedroncutpic}
    \end{subfigure}
\caption{Decomposing the 2-crossing-centered octahedra into tetrahedra to form the face-centered bipyramid decomposition.}
\label{octa}
\end{figure}

In Section 2, given any multicrossing projection of $L$, we develop a dual pair of bipyramid decompositions of the complement of $L$. These are the multicrossing generalizations of the 2-crossing projection decompositions into octahedra at the crossings and into face-centered bipyramids described above.

In Section 3, we consider upper bounds on hyperbolic volume that these decompositions yield in the case when $L$ is hyperbolic. In a subsequent paper, a more in depth analysis of the resulting volume bounds will appear.

In Section 4, we apply these upper bounds, together with properties of the infinite square weave, in order to establish two new planar tiling weaves that contain multicrossings. These are the \textit{triple weave}, consisting of 3-crossings, and the \textit{right triangle weave}, consisting of 2- and 4-crossings. Like the square weave, the triple weave and the right triangle weave realize the maximum possible volume per crossing in any link with these types of crossings. Somewhat surprisingly, the complement manifolds of the minimal finite representations of these three weaves in a thickened torus are homeomorphic to one another.

We would like to thank the referees, who substantially helped to improve the readability of this paper, especially with regard to the proof of Theorem 4.

\section{The Construction}

We first develop the \textit{face-centered bipyramid decomposition}, which holds for all multicrossing projections of all links. The face-centered bipyramids in this decomposition have finite top and bottom vertices and ideal equatorial vertices. We then demonstrate how, as in the 2-crossing case, these face-centered bipyramids can be cut into tetrahedra and reglued into a dual decomposition of the complement into crossing-centered bipyramids.

\subsection{Face-Centered Bipyramid Construction}

%(Portions of the old construction were removed here)

\begin{figure}[h]
    \begin{subfigure}[t]{0.3\textwidth}
        \centering
        \includegraphics[scale = .20]{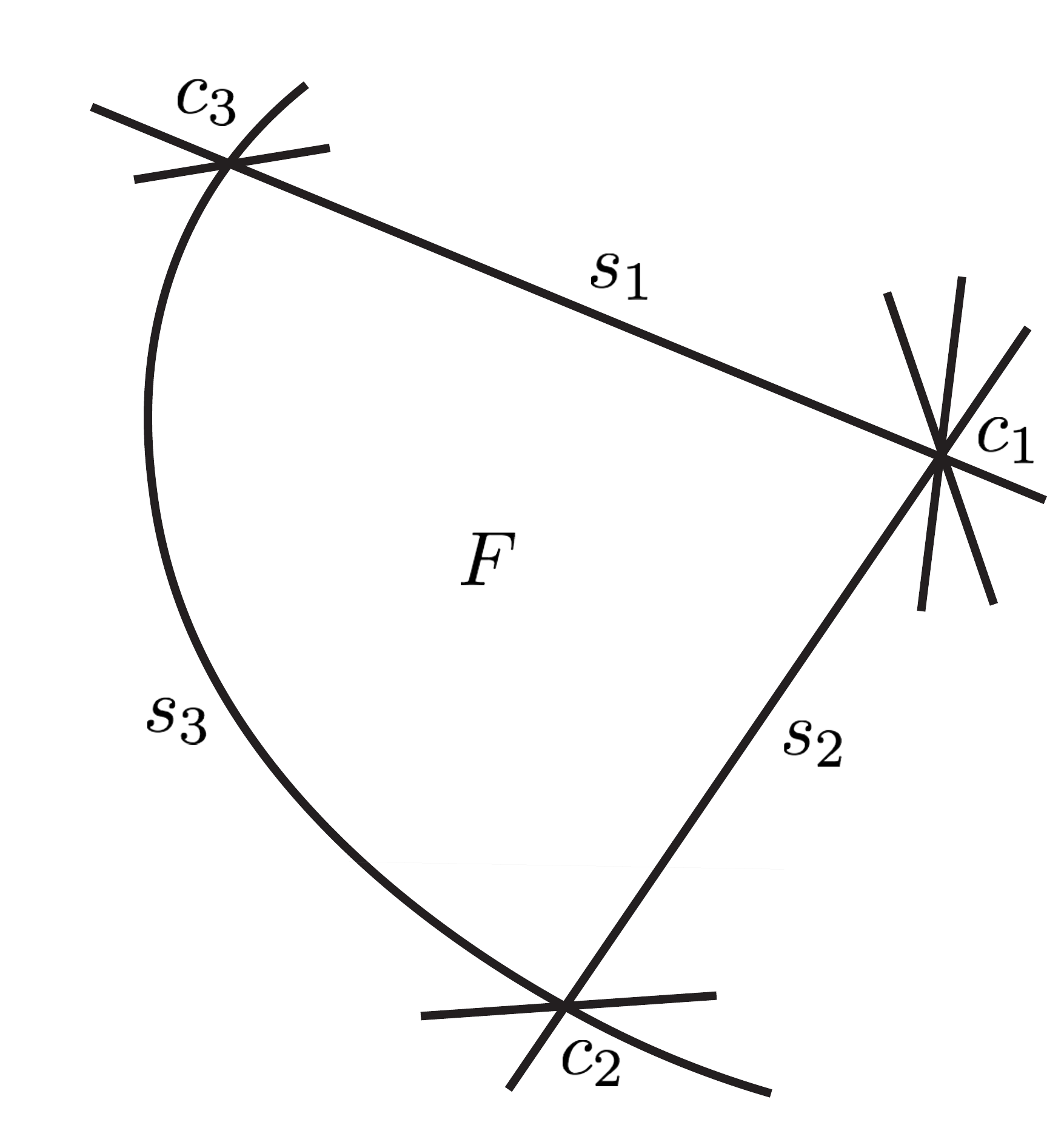}
        \caption{A face of the projection graph.}
        \label{step1face}
    \end{subfigure}
    \hfill
    \begin{subfigure}[t]{0.3\textwidth}
        \centering
        \includegraphics[scale = .20]{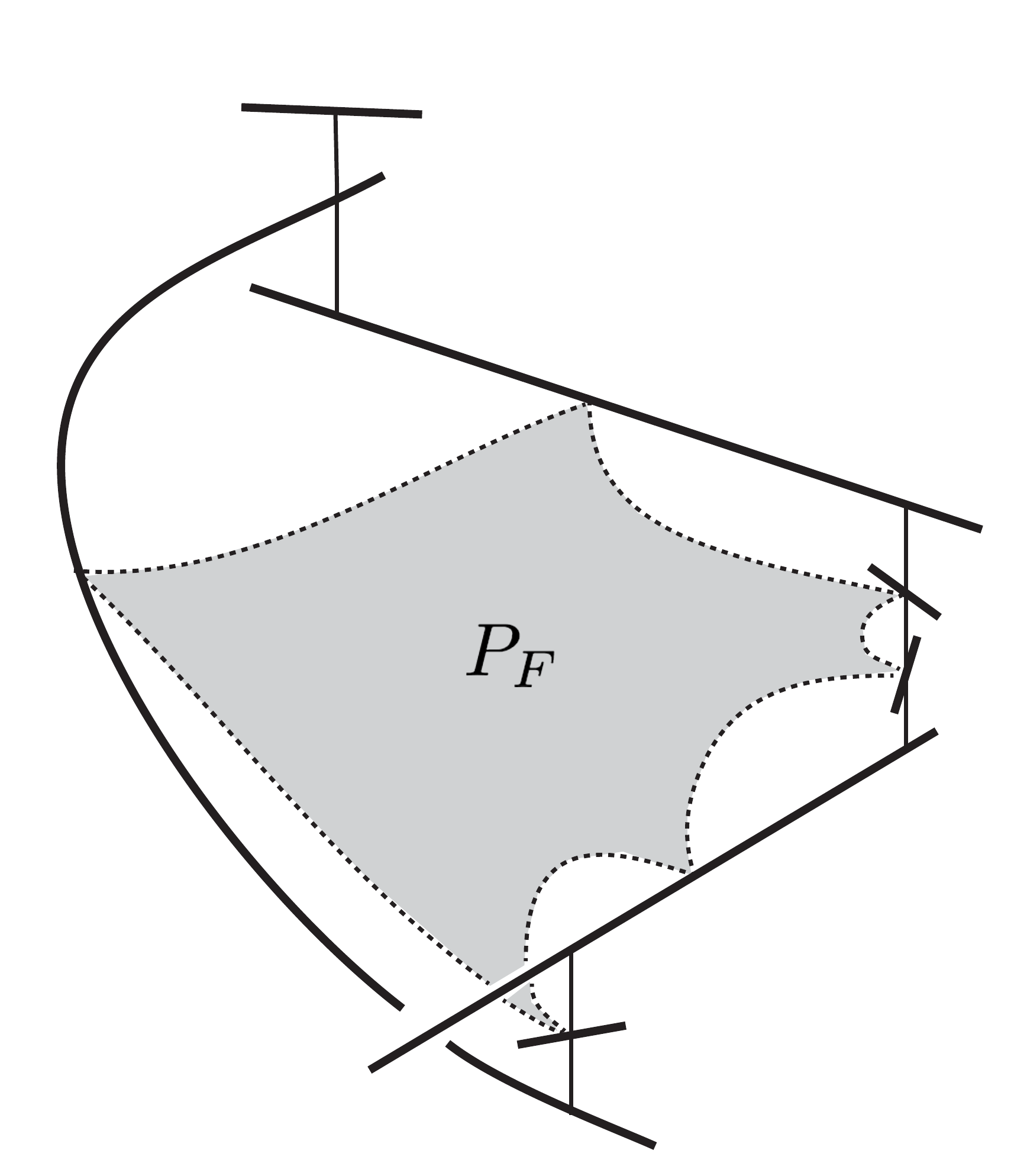}
        \caption{The equatorial edges of a face-centered bipyramid.}
        \label{step1}
    \end{subfigure}
    \hfill
    \begin{subfigure}[t]{0.3\textwidth}
        \centering
        \includegraphics[scale = .20]{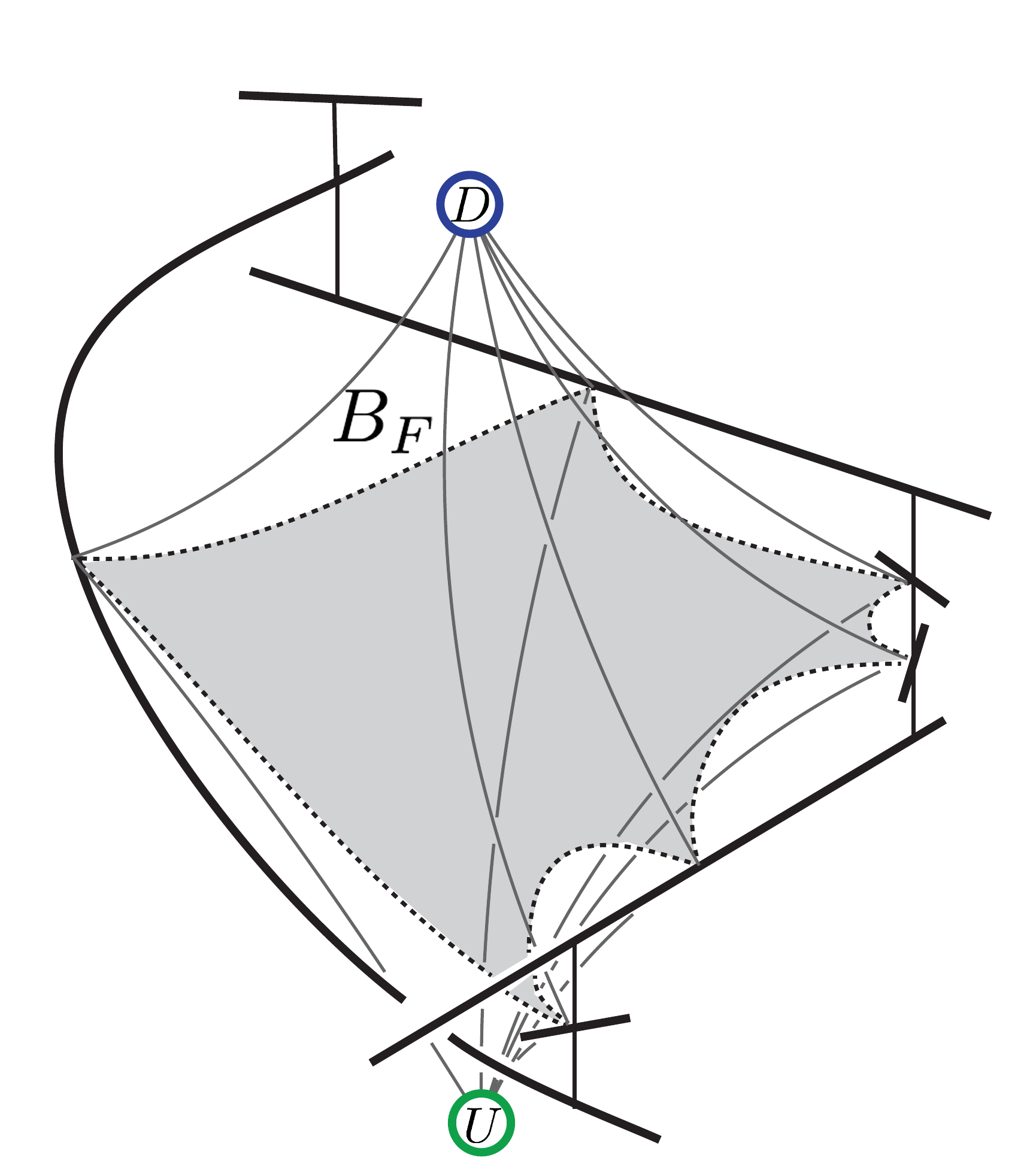}
        \caption{The face-centered bipyramid.}
        \label{step1bip}
    \end{subfigure}

    \caption{The construction of a face-centered bipyramid.}
    \label{facecentered}
\end{figure}

 We begin with a particular fixed multi-crossing projection of $L$ and associated projection graph $G$. In each face $F$ of $G$, such as in Figure \ref{facecentered}(A), we produce a cycle of edges, shown dashed, bounding a polygon $P_F$, as in Figure \ref{facecentered}(B). Each strand of the link around the boundary of $F$ contributes a vertex for $P_F$ and at each crossing on the boundary of $F$, each strand that is between the heights of the two edges of the face at that crossing also contributes a vertex of $P_F$. Then we add a finite vertex $U$ above the projection plane and a finite vertex $D$ below the projection plane, and cone each polygon up to $U$ and down to $D$. The result is a bipyramid $B_F$ corresponding to each face of the projection, including the outermost face, as in Figure \ref{facecentered}(C). We call $P_F$ the equatorial polygon of $B_F$.
 
 We now describe how to glue the faces of these bipyramids together to fill the complement of the link. We depict the construction in detail for a triple-crossing, but the general case is similar. In Figure \ref{triplecentered}(A), we see the top view of a triple crossing with part of each equatorial polygon corresponding to the adjacent face-centered bipyramids. We have also added in the top edges of the bipyramids, which meet at $U$. In Figure \ref{triplecentered}(B), we see a side view of the same crossing.

  \begin{figure}[hb]
\begin{subfigure}[t]{0.475\textwidth}
        \centering
        \includegraphics[scale = .9]{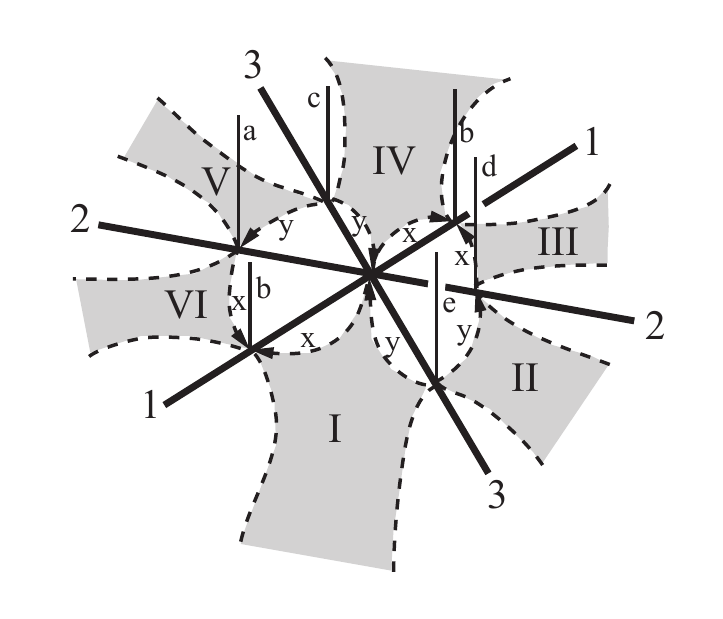}
        \caption{The view from above of how the face-centered bipyramids around a 3-crossing fit together.}
        \label{step2}
\end{subfigure} \hfill
\begin{subfigure}[t]{0.475\textwidth}
    	\centering
       \includegraphics[scale = 0.8]{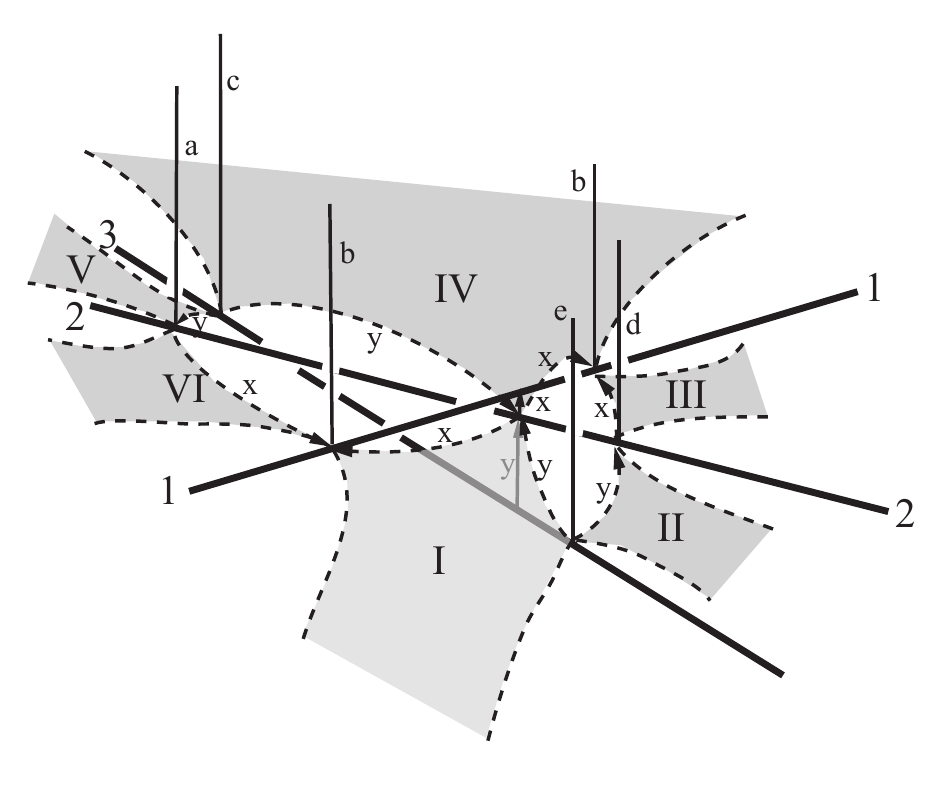}
        \caption{The view from the side of how the face-centered bipyramids around a 3-crossing fit together.}
       \label{step3}
\end{subfigure}
  \caption{The face-centered bipyramid decomposition.}
  \label{triplecentered}
  \end{figure}
  
  Here we can see how the various faces of the top half of the bipyramids glue together. For instance, representing faces by the edge classes on their boundaries, we see face $bxd$ of $I$ will glue to face $dxb$ of $III$. And face $dye$ of $I$ glues to face $eyd$ of $II$. Note that a vertical edge can slide along its link strand if there is no obstruction from another strand to doing so-this is why the two vertical edges coming out of strand 1 are both labelled with $b$.
  
   In general, consider any face-centered bipyramid $B_F$ adjacent to a multicrossing $c$, with $F$ bounded by strands that enter $c$ at levels $j$ and $k$, where the faces and strands about the multicrossing are considered in clockwise order. Consider an upper face $pqr$ of $B_F$, where $p$ extends from $U$ down to the strand in $c$ at level $i+1$, $q$ lies between the strand at level $i+1$ and the strand at level $i$, and $r$ extends from the strand at level $i$ back up to $U$, where $j\leq i<i+1\leq k$. This face will glue to its partner face $rqp$ of $B_F'$, where $B_{F'}$ is the next face-centered bipyramid encountered such that $F'$ is bounded by strands entering $c$ at levels $k'$ and $j'$, where $j' \leq i < i+1 \leq k'$. All upper faces will either be one such $pqr$ or the $rqp$ partner of some $pqr$, and so all upper faces will be paired and glued to fill space above the link.

 Finally, the faces of the bottom halves of the face-centered bipyramids glue together in a similar fashion, filling the entire complement of the link.

The discussion above is summarized in the following theorem:

\begin{thm}\label{facebip}
  Given any link $L$ and any multicrossing projection of $L$ with projection graph $G$, the complement of $L$ can be decomposed into a collection of bipyramids $\mathcal{B}_f = \{B_F :~ F \text{ a face of $G$}\}$. Furthermore, the size of each $B_F \in \mathcal{B}_F$ is given by
\begin{equation}\label{facesize}
|B_F|=\sum_{c_i \in \partial F} \left| l(s_i,c_i)-l(s_{i+1}, c_i) \right|,
\end{equation}
where $\partial F=s_1, c_1,\dots, s_m,c_m$ is the boundary of $F$ and $l(s,c)$ is the level at which strand $s$ enters crossing $c$.
\end{thm}

Note that in general a multi-crossing face-centered bipyramid may consist of more tetrahedra than there are edges bounding the face, whereas in the 2-crossing case $|B_F|$ necessarily equals the number of edges of $F$.

\subsection{Crossing-Centered Bipyramid Construction}

\begin{figure}[t]
\centering
    \begin{subfigure}[t]{0.475 \textwidth}
        \centering
        \includegraphics[scale=0.12]{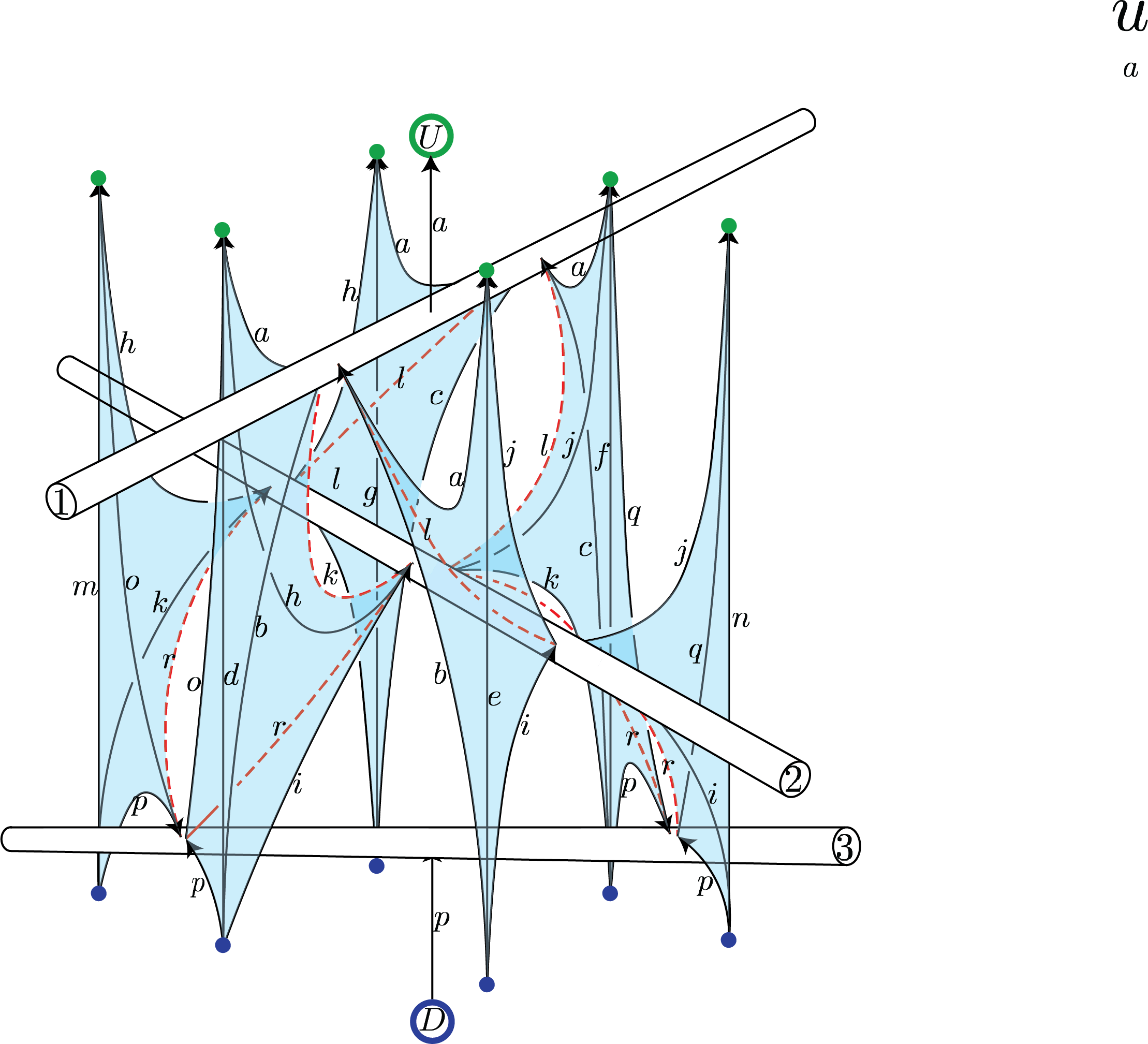}
        \caption{The 3-crossing face-centered decomposition of a link complement.}
        \label{3crossingf}
    \end{subfigure}
    \hfill
    \begin{subfigure}[t]{0.475 \textwidth}
        \includegraphics[scale=0.12]{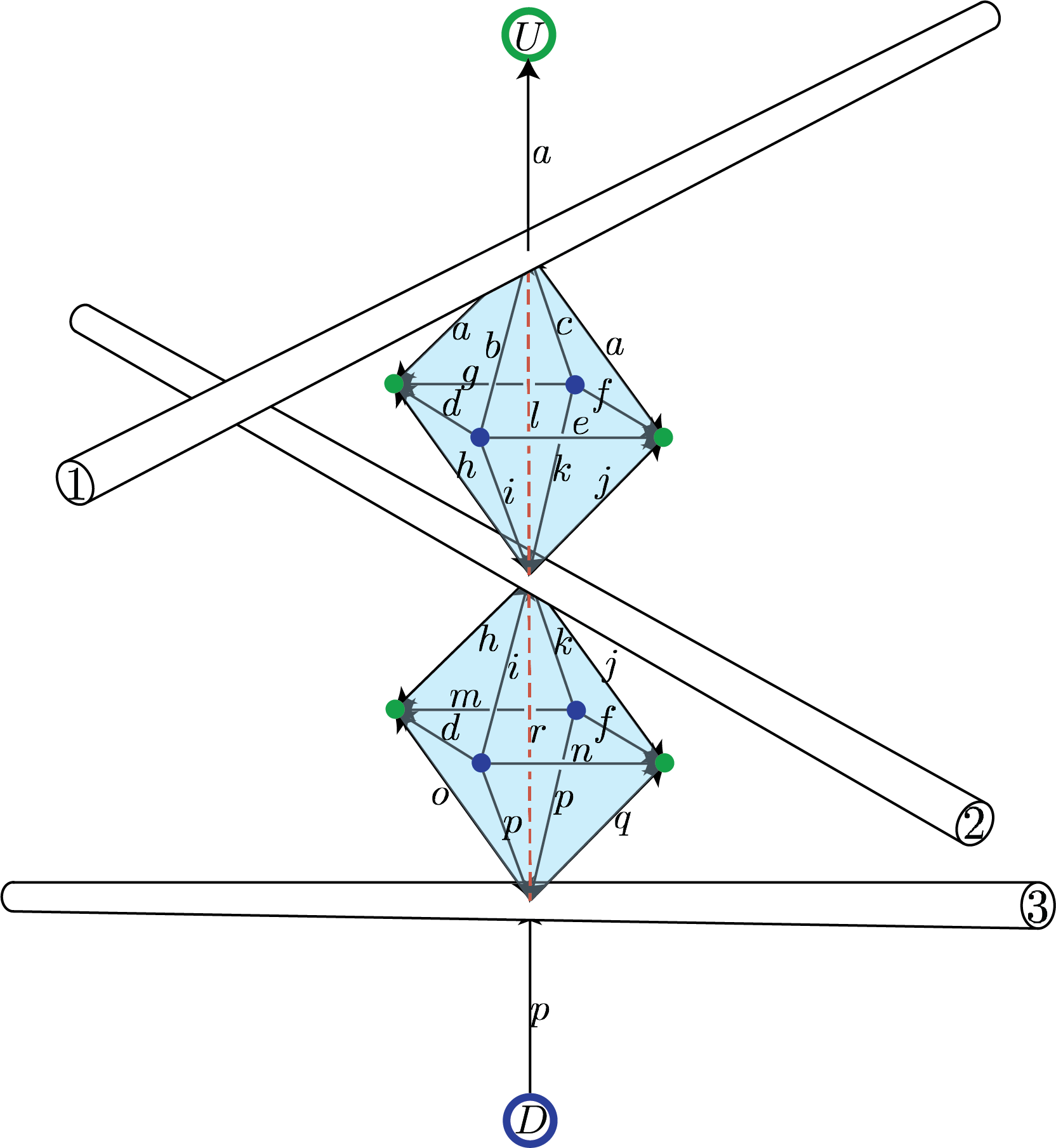}
        \caption{The 3-crossing crossing-centered decomposition of a link complement.}
        \label{3crossingc}
    \end{subfigure}
    \caption{The crossing-centered and face-centered bipyramid decompositions.}
    \label{3crossingbips}
\end{figure}

From the face-centered decomposition we derive the \textit{crossing-centered bipyramid decomposition} by first cutting each face-centered bipyramid $B_F$ into its constituent $|B_F|$ tetrahedra. These tetrahedra share an edge from $U$ to $D$ in the center of $B_F$, and the opposite edge of each tetrahedron lies between two adjacent-level strands in a multicrossing in $\partial F$. For a given $c$, consider the tetrahedra from all face-centered bipyramids neighboring $c$ which have an edge passing between adjacent-level strands of $c$. The 3-crossing case is shown in Figure \ref{3crossingbips}(A) and (B). We regroup these tetrahedra according to which of the two adjacent-level strands in $c$ at levels $i$ and $i+1$ that they touch, and so for each $i$, the edge between the level $i$ and $i+1$ strands is shared by all of the tetrahedra in this group. We can glue each of these groups of tetrahedra together about this shared edge to form bipyramid with top vertex at the level $i$ strand and bottom vertex at the level $i+1$ strand. The size of this crossing-centered bipyramid between levels $i$ and $i+1$ of $c$ is determined by the number of faces neighboring $c$ that contribute tetrahedra to it, since each such face contributes exactly one. This is captured by the following theorem:

\begin{thm}\label{crossingbipsize}
  Given any link $L$ and any projection $P$ of $L$ with multicrossings $C$, the complement of $L$ can be decomposed into a collection of bipyramids $\mathcal{B}_C = \{B_{c, i} :~ c \in C, i\in \{1, \dots , |c| -1 \}\}$. Furthermore, the sizes of these bipyramids are given by
  
\begin{equation}\label{crossingbipsizeeqn}
  |B_{c,i}|=2\left| \left\{j \in \{1, \dots, |c|\}: \min\{l_j, l_{j+1}\} < i + 1/2 < \max\{l_j,l_{j+1}\} \right\} \right|,
\end{equation}
where $c$ is a multicrossing composed of the strands $s_1, \dots, s_j, \dots, s_n$ at crossing levels $l_1, \dots l_j, \dots, l_n$.

\end{thm}

\begin{proof}
  The crossing-centered bipyramid $B_{c,i}$ can be cut into a collection of $|B_{c,i}|$ tetrahedra that share the bipyramid's central edge and glue face-to-face around it. Each of these tetrahedra comes from exactly one face-centered bipyramid $B_F$, where $c$ is in $\partial F$. By the construction of the face-centered bipyramids above, if the boundary of $F$ is of the form $\partial F = \dots , s_{j+1}, c, s_j, \dots$, then $B_F$ contributes a tetrahedron to $B_{c,i}$ exactly when either $l_{j+1} < i + 1/2 < l_j$ or $l_{j+1} > i + 1/2 > l_j$. Therefore as the adjacent pairs of strands at levels $i$ and $i+1$ of $c$ are considered in turn, for the two face-centered bipyramids in the two faces bounded by $s_j$ and $s_{j+1}$ and opposite $c$ from one another, either both face-centered bipyramids contribute a tetrahedron to $B_{c, i}$ or neither does.
  
This shows that these collections of tetrahedra are of the stated size; it remains to show that they glue together to form bipyramids. But the gluings that merge these tetrahedra into bipyramids are exactly the gluings that describe how the face-centered bipyramids glue up to fill the complement of $L$. The pairs of triangular faces that meet around the central edge of each face-centered bipyramid are alternating pairs of the partnered upper faces and lower faces of the face-centered bipyramids surrounding $c$. In the construction of the crossing-centered bipyramids from the face-centered bipyramids, the equatorial edges of each become the central edges of the other.
\end{proof}

\begin{figure}[htpb]
\centering
    \begin{subfigure}[b]{0.3\textwidth}
    	\centering
        \includegraphics[scale = 0.4]{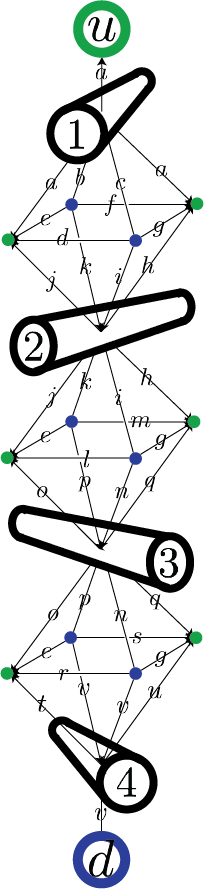}
        \caption{$1234$}
        \label{1234}
    \end{subfigure}
    ~
    \begin{subfigure}[b]{0.3\textwidth}
    	\centering
        \includegraphics[scale=0.4]{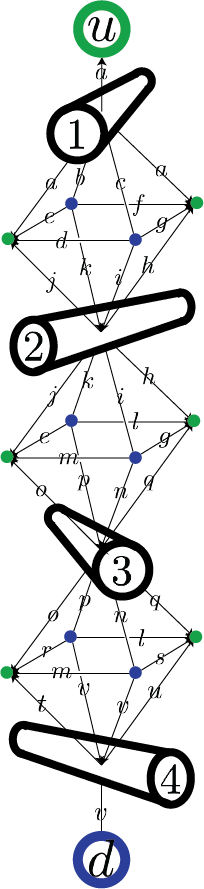}
        \caption{$1243$}
        \label{1243}
    \end{subfigure}
    ~
    \begin{subfigure}[b]{0.3\textwidth}
        \centering
        \includegraphics[scale = 0.4]{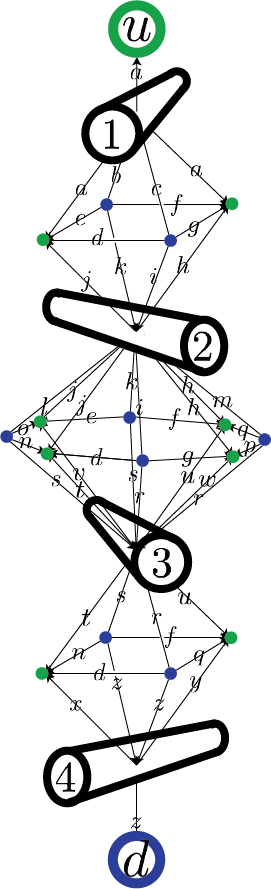}
        \caption{$1324$}
        \label{1324}
    \end{subfigure}
\caption{The six 4-crossing configurations, identified up to reflection by their permutations, and their crossing-centered bipyramid decompositions, with edges labelled by their edge classes.}
\label{bips}
\end{figure}

For convenience, the criterion for whether a given face-centered bipyramid $B_F$ with face boundary $\partial F = \dots, s_j, c, s_{j+1}, \dots$ contributes a tetrahedron to the crossing-centered bipyramid $B_{c, i}$ can be reframed in terms of interval containment in the following way. The interval $[i, i+1]$ represents the position of $B_{c, i}$ in the crossing, and $[l_j, l_{j+1}]$ represents the range of levels of $c$ that are spanned by $B_F$. Then $B_F$ contributes a tetrahedron to $B_{c,i}$ if and only if $[i, i+1] \subseteq [l_j, l_{j+1}]$.

\begin{coro}\label{boundedbipgrowth}
  For any multicrossing $c$, $|B_{c, 1}| = |B_{c, |c|-1}| = 4$ and the sizes of adjacent crossing-centered bipyramids must satisfy
$$\left||B_{c, i}| - |B_{c,i+1}|\right| = 0 \hspace{.1in} \text{or} \hspace{.1in} 4$$
\end{coro}

\begin{proof}
To see that $|B_{c,1}| = |B_{c, |c|-1}| = 4$, first consider $B_{c,1}$. If the top strand in $c$ is $s_i$ with level $l_i=1$, then $[l_{i-1}, l_i]$ and $[l_i, l_{i+1}]$ are the only adjacent-strand level intervals containing $[1,2]$. Therefore by Theorem \ref{crossingbipsize}, $B_{c,1}$ is composed of 4 tetrahedra glued face-to-face and sharing a common edge, and it is therefore an octahedron. The bottom bipyramid $B_{c, |c|-1}$ is an octahedron for the same reason.

Within $c$, the sizes of two neighboring crossing-centered bipyramids $B_{c,i}$ and $B_{c,i+1}$ correspond to the frequency with which the intervals $[i,i+1]$ and $[i+1, i+2]$ are contained in the strand level intervals $[l_j, l_{j+1}]$. If $s_k$ is the strand with $l_k=i+1$ that passes between these two bipyramids, then the intervals $[l_j, l_{j+1}]$ will contain either both $[i,i+1]$ and $[i+1, i+2]$ or neither, unless $s_j=s_k$ or $s_{j+1} = s_k$. Therefore, the difference between $|B_{c,i}|$ and $|B_{c,i+1}|$ is determined by how the tetrahedra are allocated from the four face-centered bipyramids around $c$ that are bordered by $s_k$ and correspond to the strand level intervals $[l_{k-1}, i+1]$ and $[i+1, l_{k+1}]$. If $l_{k-1} < l_k$ and $l_{k+1} < l_k$, then these four face-centered bipyramids will contribute four tetrahedra to $B_{c,i}$ and none to $B_{c,i+1}$, and $|B_{c,i}| = |B_{c,i+1}|+4$. If $l_{k-1} > l_k$ and $l_{k+1} > l_k$, then these four tetrahedra will be allocated to $B_{c,i+1}$, and $|B_{c,i}| = |B_{c,i+1}|-4$. And if $l_{j-1} < l_j < l_{j+1}$ or $l_{j-1} > l_j > l_{j+1}$, then the contributions to the two bipyramids will be the same, so $|B_{c,i}|=|B_{c,i+1}|$.
\end{proof}

In Figure \ref{bips}, we see the crossing-centered bipyramid decompositions for all six 4-crossing configurations (identified up to reflection). This includes the first instance of a non-octahedral crossing-centered bipyramid, shown in Figure \ref{bips}(C). In light of the constraints on crossing-centered bipyramid sizes given by Corollary \ref{boundedbipgrowth}, it is natural to ask which sequences of crossing-centered bipyramid sizes are realizable. It turns out that these conditions constitute a classification of the realizable crossing-centered bipyramid size sequences:

\begin{thm}\label{xingalg}
Every sequence $m_1, m_2, \dots,m_{n-1}$ of positive integers such that \\ $m_1=m_{n-1}=4$ and $|m_i- m_{i+1}| = 0 \hspace{.1in}\text{or} \hspace{.1in}4$ is realized as the signature of crossing-centered bipyramid sizes for some $n$-crossing.
\end{thm}

In order to prove this theorem, we first prove two lemmas.

\begin{lem}\label{add4} If $m_1, m_2, \dots,m_{n-1}$ is realized, so is $4, m_1+4, m_2+4, \dots,m_{n-1}+4, 4$.
\end{lem}

\begin{proof} Let $c$ be the $n$-crossing that realizes $m_1, m_2, \dots,m_{n-1}$ with level sequence $l_1,l_2,\dots,l_n$. We create an $(n+2)$-crossing $c'$ by adding a strand above and below $c$ in the following manner. Add the new overstrand clockwise from the understrand of $c$ and the new understrand just clockwise from the new overstrand. The contributions of the intervals $[l_j, l_{j+1}]$ to the sizes of the bipyramids between strands remain unchanged, except for in three cases. Moving clockwise from the old understrand of $c$, the interval between the old understrand of $c$ and the new overstrand of $c'$ will contribute to every bipyramid except the bipyramid above the very bottom strand of $c'$. The interval between the new overstrand and the new understrand will contribute to every bipyramid between strands in $c'$. And the interval between the new understrand of $c'$ and the strand that was clockwise from the understrand in $c$ will contribute to the same set of bipyramids it did before, as well as to the bottom bipyramid above the new understrand. Since each such contribution is doubled when we consider the intervals on the opposite side of the crossing, this means that the sequence of bipyramid sizes for $c'$ is  $4, m_1+4, m_2+4, \dots,m_{n-1}+4, 4$.
\end{proof}

\begin{lem}\label{concatenate} If sequences $p_1, p_2, \dots,p_{u-1}$ and $q_1, q_2, \dots,q_{v-1}$ are realized, then so is %
$p_1, p_2, \dots,p_{u-1}, q_2, \dots, q_{v-1}$.
\end{lem}

\begin{proof} Note that both realized sequences begin and end with 4's, so in the concluding sequence, the last 4 of the first sequence given by $p_{u-1}$ is identified with the beginning 4 of the second sequence, given by $q_1$.  Let $c_1$ and $c_2$ be multicrossings realizing the two given sequences. Construct a new $(n+q-2)$-crossing $c_3$ by starting with the first crossing and then placing directly beneath it the second crossing, such that from above, the entire second crossing appears in the two opposite regions just clockwise from the bottom strand in $c_1$. Moreover, do so such that the topmost strand of $c_2$  is clockwise from the bottom strand of $c_1$. (See Figure \ref{concatenation}.) Now remove the bottom strand of $c_1$ and the top strand of $c_2$ to obtain our new crossing $c_3$.

\begin{figure}[htpb]
\centering
           \includegraphics[scale=.8]{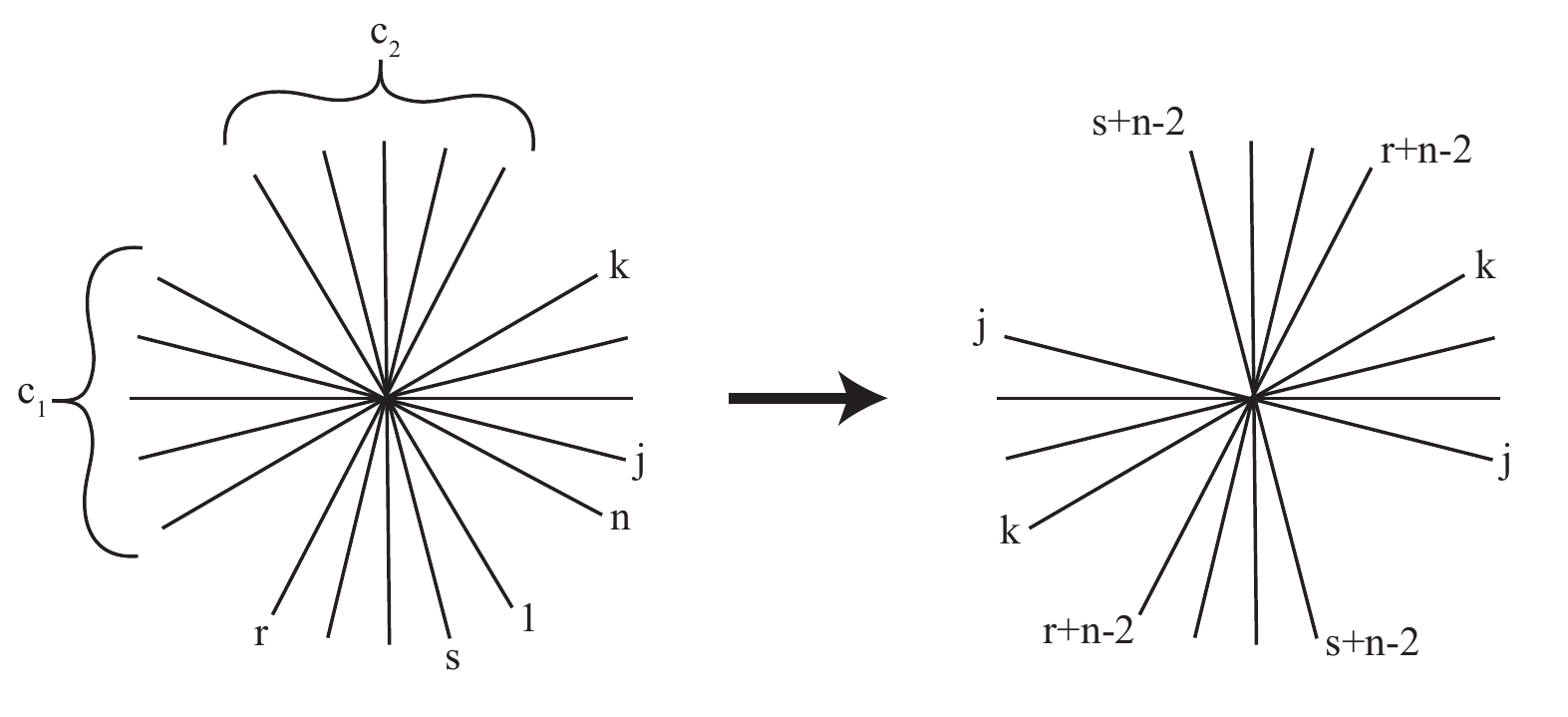}
\caption{Creating crossing $c_3$ from $c_1$ and $c_2$ to``concatenate" bipyramid size seqences.}
  \label{concatenation}
\end{figure}

        Let $j$ and $k$ be the heights of the strands counterclockwise and clockwise from the bottom strand in $c_1$. Let $r$ and $s$ be the heights of the strands  counterclockwise and clockwise from the top strand in $c_2$. In $c_3$, the only intervals originally from $c_1$ with their contributions to bipyramids between strands affected are $[j,n]$ and $[k,n]$. Similarly, the only intervals originally from $c_2$ with their contributions to bipyramids between strands affected are $[1,r]$ and $[1,s]$. In $c_3$, we also have the new intervals $[j, s+n-2]$ and $[k, r+ n-2]$. Then the contributions to the sizes of bipyramids by $[j,n]$ and $[1,s]$ in $c_1$ and $c_2$ are exactly replaced by the contributions from $[j, s+n-2]$, with the exception that there is a single intermediate bipyramid that is contributed to rather than separate bipyramids at the bottom of $c_1$ and the top of $c_2$.

   The same holds for replacing the contributions of intervals $[k,n]$ and $[1,r]$ by \\$[k, r+n-2]$. Hence $c_3$ realizes the desired sequence.
\end{proof}

\begin{proof}[Proof of Theorem \ref{xingalg}]
  We induct on the sum of the bipyramid size sequence $\sum_{i = 1}^{n-1} m_i$ in a sequence $\{m_i\}$. We can realize the single integer sequence $\{4\}$ with a 2-crossing. Suppose that we can realize all sequences $\{m_i\}$ such that 
    \begin{enumerate}
        \item $m_1 = m_{n-1} = 4 \label{4end}$,
        \item $|m_i - m_{i+1}| \leq 4$ for all $i$ \label{smallgap}, and
        \item $\sum_{i = 1}^{n-1} m_i \leq 4t$ \label{smallsum}.
    \end{enumerate}

    Then given a sequence $\{m_i\}$ satisfying (\ref{4end}) and (\ref{smallgap}) and for which $\sum_{i = 1}^{n-1} m_i = 4(t+1)$, either $\{m_i\}$ contains a 4 that is not at the beginning or end, or it does not. If it does, then $\{m_i\}$ is of the form $p_1, \dots , p_{k-1}, q_2, \dots, q_{l-1}$ for two sequences $\{p_i\}$ and $\{q_i\}$ that both satisfy (\ref{4end}), (\ref{smallgap}), and (\ref{smallsum}). These sequences are therefore realizable, and so $\{m_i\}$ is realizable by Lemma \ref{concatenate}.
  
    If $\{m_i\}$ does not contain a 4 in its interior, then it is of the form $4, p_1, 4+p_2, \dots 4+p_{n-1}, 4$ for some sequence $\{p_i\}$ satisfying (\ref{4end}), (\ref{smallgap}), and (\ref{smallsum}). Therefore $\{p_i\}$ is realizable, and so by Lemma \ref{add4} $\{m_i\}$ is realizable as well.

\end{proof}

%Corresponds to cut figure (previously Figure 6)
%This algorithm is illustrated in \ref{algpic} for an arbitrary example, with a 12-crossing and the crossing-centered bipyramid size signature $4,8,12,12,12,8,12,8,12,8,4$ given as inputs.

This crossing-centered bipyramid decomposition agrees with the construction used in \cite{Adams2012} Theorem 5.2, which shows that for a link $L$ in a 3-crossing projection, the complement can be decomposed into pairs of octahedra positioned between the strands of each 3-crossing, as in Figure \ref{3crossingbips}(B).

\section{Hyperbolic Volume Bounds}

In hyperbolic space $\mathbb{H}^3$, for each fixed $n$ there is a maximum $n$-bipyramid volume. Therefore, given a decomposition of a hyperbolic link complement $S^3\setminus L$ into bipyramids, the volume
of the entire manifold $S^3\setminus L$ is bounded above by the sum of the maximum possible volumes of each of its constituent bipyramids. We pursue this strategy in order to develop upper bounds for volumes of hyperbolic link complements,  given their multicrossing projections and the corresponding bipyramid decompositions developed in Section 2. To begin, we know from  \cite{Adams2015} Theorem 2.2 that the volumes of these maximal size-$n$ bipyramids, here denoted $B_n$, grow logarithmically in $n$:
%Should this be a lemma in this context, given that it informs the bounds that follow?
\begin{thm} \label{logbound}
$vol(B_n)<2\pi\log(n/2)$ for $n\geq 3$ and $vol(B_n)$ grows asymptotically like $2\pi\log(n/2)$:
$$\lim_{n\rightarrow \infty} \frac{vol(B_n)}{2\pi\log(n/2)}=1.$$
\end{thm}

    In a subsequent paper, we will explore the volume bounds to which these decompositions give rise and consider techniques for improving upon these bounds.

    Here we note that the multicrossing-centered bipyramid decomposition for a hyperbolic link $L$ in a multi-crossing projection, the derived crossing-centered bipyramids $\mathcal{B}_c$ give an upper bound on the volume of $L$. This bound is
\begin{equation}\label{MCCB}
  vol(L) < \sum_{B \in \mathcal{B}_C} vol(B_{|B|}),
\end{equation}
where for $B \in \mathcal{B}_C$, $|B|$ is given by Theorem\ref{crossingbipsize}. This multicrossing crossing-centered bipyramid bound will be referred to as the MCCB bound on volume.

Similarly, the multicrossing face-centered bipyramid decomposition also gives us an upper bound on volume. For a given link $L$ in a multi-crossing projection with derived face-centered bipyramids $\mathcal{B}_F$, this bound is
%Switch B_f (here and above) to N(f) (or something similar) to more closely mimic the MCCB bound?
\begin{equation} \label{MFCB}
  vol(L)< \sum_{B \in \mathcal{B}_F} vol(B_{|B|}),
\end{equation}
where for $B \in \mathcal{B}_C$, $|B|$ is given by Theorem \ref{facebip}. This multicrossing face-centered bipyramid bound will be referred to at the MFCB bound on volume.

Note that for both the face-centered and crossing-centered bipyramid decompositions, the specific configurations of the $n$-crossings affect the sizes of the bipyramids, and as a result the MCCB and MFCB bounds depend on the crossing configurations. Certain crossing configurations yield larger volumes and volume upper bounds than others. This variation will be investigated further in the subsequent paper. Note also that if we apply the Thurston 2-crossing octahedral upper bound on volume to a multi-crossing projection, then for each $n$-crossing it gives an upper bound of ${n \choose 2} v_{oct}$. This is because each $n$-crossing must be perturbed into ${n \choose 2}$ 2-crossings. On the other hand, the MCCB bound applied to an $n$-crossing gives an upper bound of $(n-1) v_{oct}$ in the best case, and an upper bound that is $O(n \log n)$ in the worst case. Table \ref{boundcomparison} compares these bounds for some values of $n$.

\begin{table}
    \centering
    \begin{tabular}{|c|c|c|c|}
        \hline
        n & Best-case MCCB bound & Worst-case MCCB bound & Octahedral bound \\ \hline
        3 & 7.32772 & 7.32772 & 10.9916 \\ \hline
        4 & 10.9916 & 15.1827 & 21.9832 \\ \hline
        5 & 14.6554 & 23.0377 & 36.6386 \\ \hline
        10 & 32.9747 & 81.6887 & 164.874 \\ \hline
        100 & 362.722  & 2,183.09 & 18,136.1 \\
        \hline
    \end{tabular}
    \medskip
    \caption{Contribution to best-case, worst-case, and octahedral upper bounds on volume from $n$-crossings.} \label{boundcomparison}
\end{table}

%
%
% Should we find a better name than "right triangle weave"?
\section{Maximal Weaves}

In \cite{CKP} Champanerkar, Kofman, and Purcell introduced the \textit{infinite weave} $\mathcal{W}$, which is the unique infinite alternating link embedded in $\mathbb{R}^3$ with the $(4^4)$ regular tiling of the Euclidean plane as its projection graph, as in Figure \ref{weave}(A). They also study the \textit{volume density} of hyperbolic links, which is defined for a link $L$ to be
\begin{equation}
\mathcal{D}_{vol}(L)=\frac{vol(L)}{c(L)},
\end{equation}
and they considered $\mathcal{W}$ as the limit of an infinite sequence of finite links that contain increasingly large patches of the square weave.

In this manner, they showed that $\mathcal{W}$ is \textit{geometrically maximal}, meaning that in the limit it attains the maximal value of $\mathcal{D}_{vol}(\mathcal{W})=v_{oct}$, which realizes D. Thurston's octahedral upper bound on volume. Since each face of $\mathcal{W}$ has 4 sides, it also realizes the face-centered bipyramid upper bound of \cite{Adams2015}.

% Add numbers in the subfigures to indicate the crossing levels for B and C?
% Figures of the three weaves.
\begin{figure}[htb!]
\centering
\begin{subfigure}{.3 \textwidth}
\centering
\includegraphics[scale=0.6]{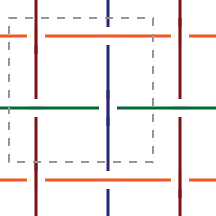}
\caption{Square weave.}
\label{squareweave}
\end{subfigure}
\quad
\begin{subfigure}{.3 \textwidth}
\centering
\includegraphics[scale=0.6]{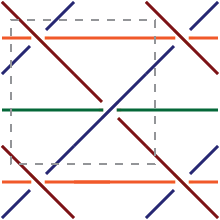}
\caption{Triple weave.}
\label{tripleweave}
\end{subfigure}
\quad
\begin{subfigure}{.3 \textwidth}
\centering
\includegraphics[scale=0.6]{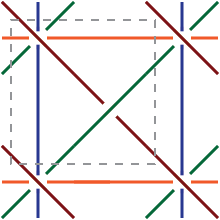}
\caption{Right triangle weave.}
\label{rightweave}
\end{subfigure}
\caption{Three weaves.}
\label{weave}
\end{figure}

We now apply these new decompositions and corresponding volume upper bounds to the \textit{triple weave} $\wt$  corresponding to the $(3^6)$ regular tiling, shown in Figure \ref{weave}(B), and to the \textit{isoceles right triangle weave} $\wrt$, corresponding to the $[4.8^2]$ Laves tiling, shown in Figure \ref{weave}(C). These are periodic infinite links embedded in $\mathbb{R}^3$. The triple weave has triple crossings of type 123 and 132 in alternate rows. The right triangle weave has an equal ratio of 2-crossings and 4-crossings, with 4-crossings given by the permutation 1243, where the top strand in the 4-crossing  passes through the 2-crossing as an understrand and the bottom strand in the 4-crossing passes through the 2-crossing as an overstrand.

For all three weaves we can take the quotient of  $\mathbb{R}^3$ by $\mathbb{Z}^2$, its discrete subgroup of translational isometries, to obtain a link in a thickened torus $T \times (0,1)$. Equivalently, we can view these as links in $S^3$, where we have added two components, each of which is a core curve of one of the solid tori to either side of the projection torus, so the complement of these two components is $T \times (0,1)$.

For the square weave, we denote this six-component link complement  in $S^3$ by $\mathcal{W}'$. There are four 2-crossings on the projection torus and the projection of the four components coming from the square weave is alternating on the torus, which is apparent from Figure \ref{tripleconstruction}. The core curves of the solid torus are shown in pink and light blue. The four link components of $\mathcal{W}'$ each bound a twice-punctured disk in the complement in $S^3$, and two of these twice-punctured disks are shown here shaded.

For the triple weave, we denote the link complement by $\wt'$. There are two 3-crossings on the torus, as in Figure \ref{torus}(B). And for the isoceles right triangle weave, the corresponding link complement is denoted $\wrt'$ and there is a single 2-crossing and a single 4-crossing on the torus, as in Figure \ref{torus}(D).

%Images of the triple weave and right triangle weave!
\begin{figure}[htb!]
\label{triplefigs}
\centering
% square weave with twice-punctured disks L, triple weave R
\begin{subfigure}{.45 \textwidth}
\centering
\includegraphics[scale=0.1]{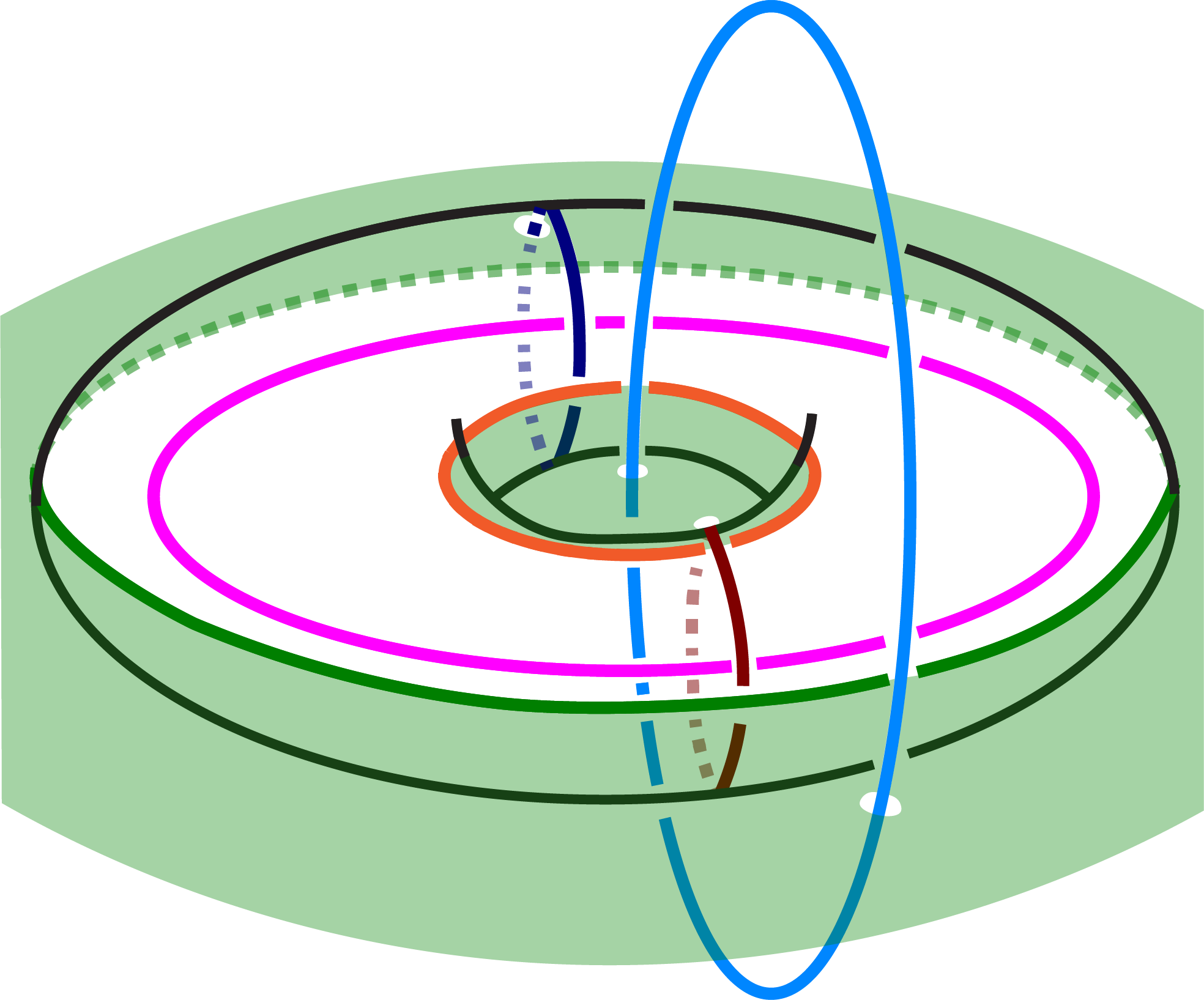}
\caption{The link complement $\mathcal{W}'$, with a pair of twice-punctured disks highlighted.}
\label{tripleconstruction}
\end{subfigure}
\quad
\begin{subfigure}{.45 \textwidth}
\centering
\vspace{32pt}
\includegraphics[scale=0.1]{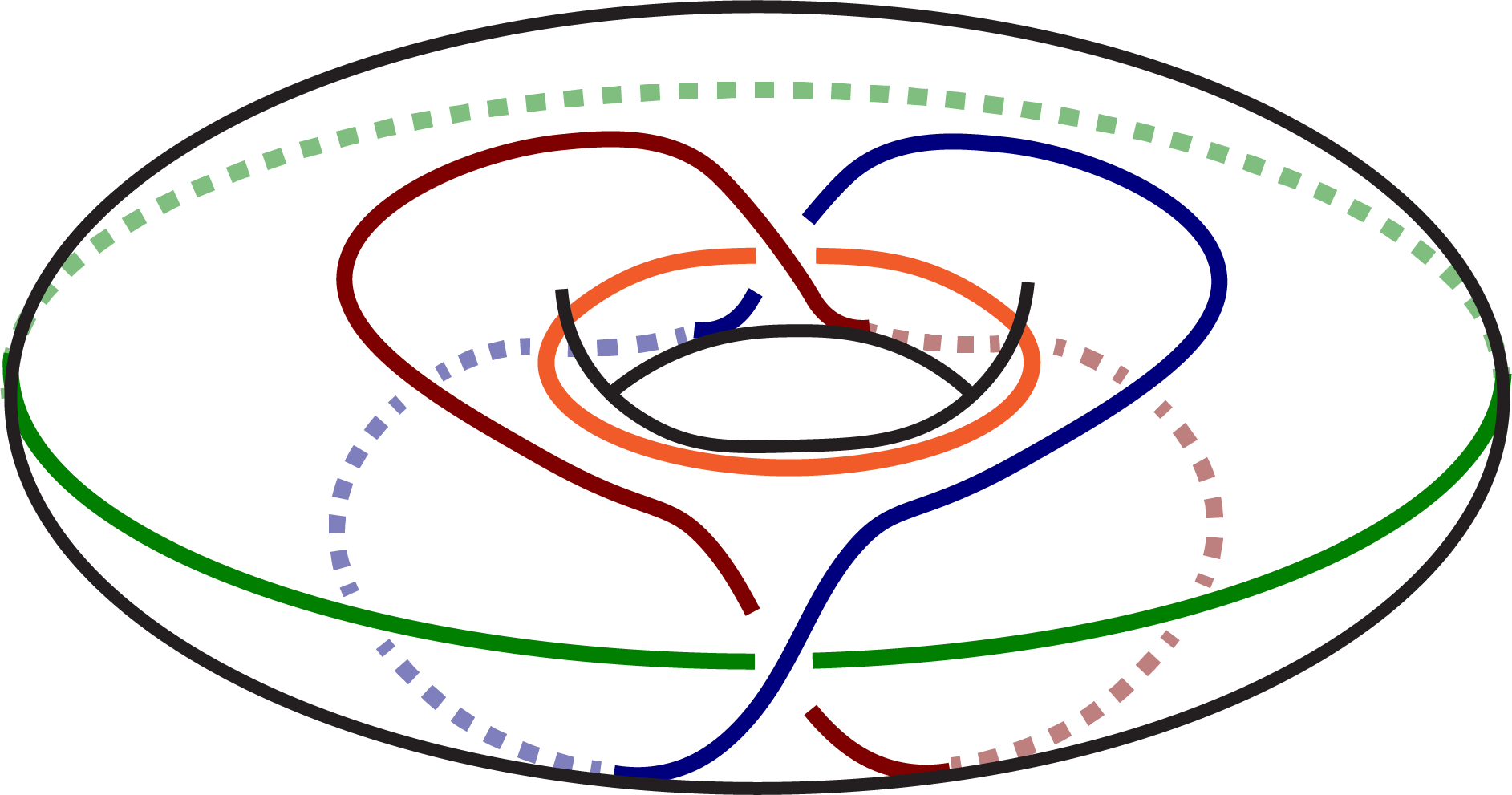}
\vspace{32pt}
\caption{The link $\wt'$ on the torus, derived from Figure \ref{torus}(A).}
\label{tripletorus}
\end{subfigure}

% square weave with twice-punctured disks L, right triangle weave R
\begin{subfigure}{.45 \textwidth}
\centering
\includegraphics[scale=0.1]{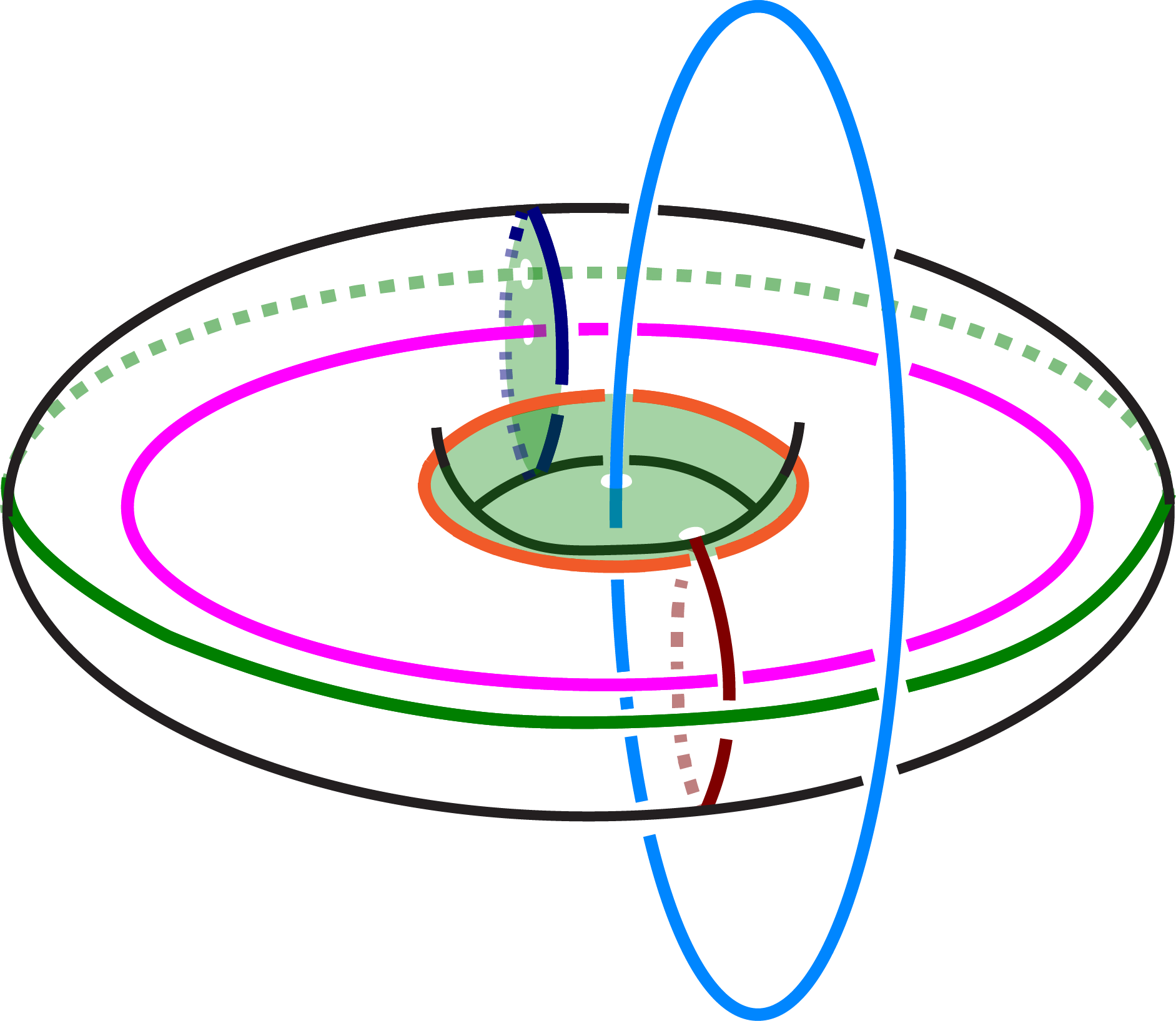}
\caption{The link complement $\mathcal{W}'$ with a different pair of twice-punctured disks highlighted.}
\label{rightconstruction}
\end{subfigure}
\quad
\begin{subfigure}{.45 \textwidth}
\centering
\vspace{22pt}
\includegraphics[scale=0.1]{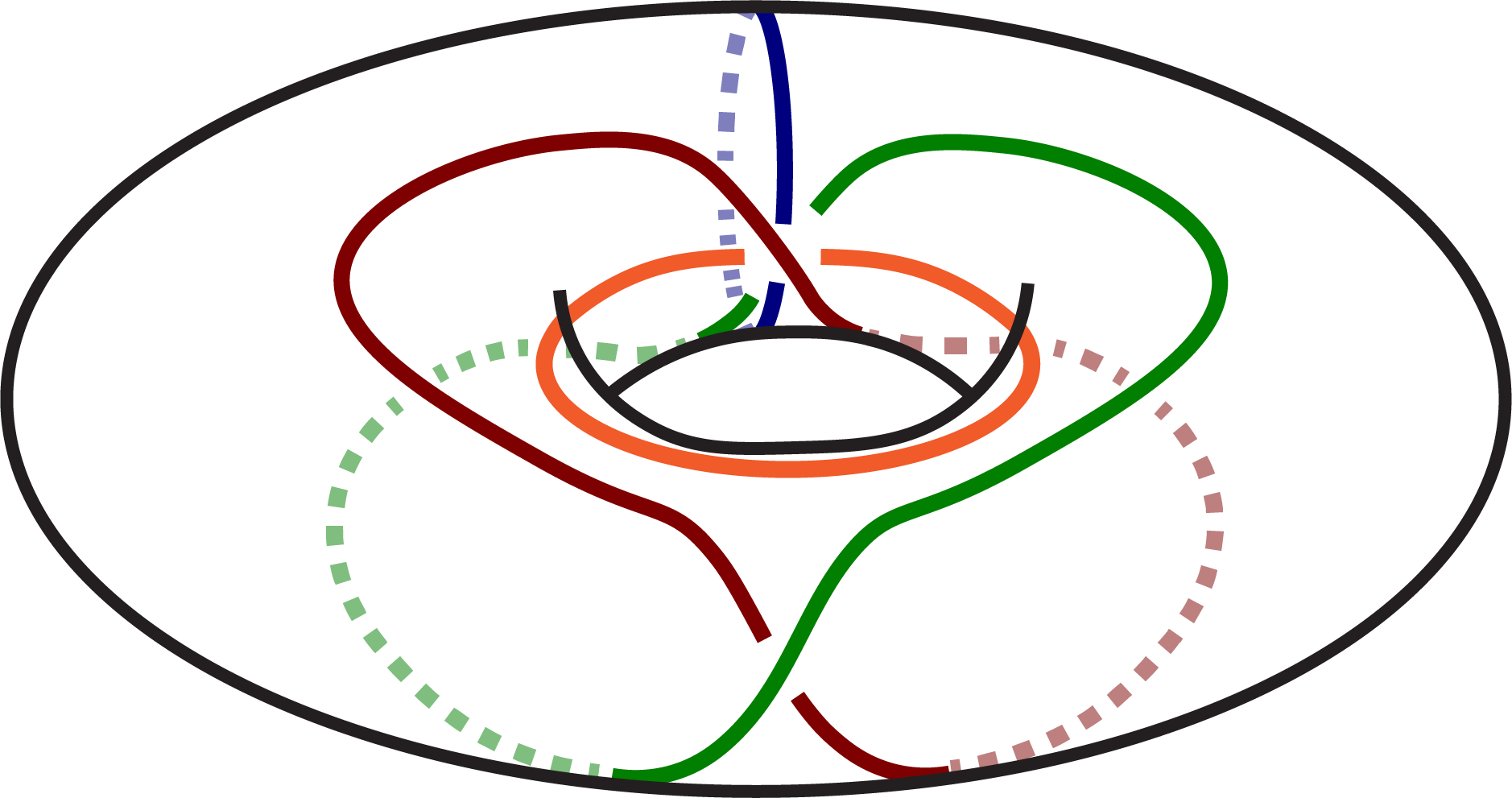}
\vspace{30pt}
\caption{The link $\wrt'$ on the torus, derived from Figure \ref{torus}(C).}
\label{righttriangletorus}
\end{subfigure}
%\begin{subfigure}{ \textwidth}
%\centering
%\includegraphics[scale=0.1]{righttorusrough}
%\caption{The link complement $\wrt'$ with two twice-punctured disks highlighted.}
%\label{righttorus}
%\end{subfigure}
\caption{On the left, two pairs of twice-punctured disks in the complement of $\mathcal{W}'$. The core curves are shown in pink and blue. On the right, the representations of $\wt'$ and $\wrt'$ generated by cutting open, twisting full twists, and regluing on these pairs of disks.}
\label{torus}
\end{figure}

%Is this precisely what they show? The maximality of the infinite object seems to follow more directly from symmetry arguments on the octahedra.

\begin{thm}\label{WT}
$\mathcal{W}', \wt'$ and $\wrt'$ are all isometric with volume equal to $4v_{oct}$.
\end{thm}

\begin{proof}
By cutting $\mathcal{W}'$ open along the two thrice-punctured spheres highlighted in Figure \ref{torus}(A), introducing a full twist on each, and regluing them, we obtain $\wt'$, as shown in Figure \ref{torus}(B). The resulting manifold is isometric to the original and therefore has the same volume.

Similarly, introducing a full twist to each of the thrice-punctured spheres highlighted in Figure \ref{torus}(C)  yields $\wrt'$ as shown in Figure \ref{torus}(D). Therefore all three are isometric.  Using either results of \cite{CKP} as applied to $\mathcal{W}'$ or the decomposition of any of these link complements into four ideal octahedra meeting four along each edge yields a volume, via the Mostow/Prasad Rigidity Theorem, of $4 v_{oct}$. \end{proof}

For all three weaves, the volume of this shared manifold achieves both the MFCB bound and MCCB bound. In the case of $\mathcal{W}'$, there is one ideal regular octahedron corresponding to each of the crossings of the link in $T \times (0,1)$. In the case of $\wt'$, there are two regular ideal octahedra at each of the two triple crossings in $T \times (0,1)$. In the case of $\wrt'$, there is one ideal regular octahedron at the 2-crossing and three octahedra at the single 1243-crossing in $T \times (0,1)$. For the MFCB bound, we consider the bipyramids coming from the faces of the projection onto the torus, and we obtain one regular ideal octahedron per face for the four faces, in each of these three cases, again realizing the upper bound on volume. Note that we must remove additional link components in order to be in $T \times (0,1)$ and make all vertices ideal on the octahedra, which is necessary in order to realize the upper bounds on volume.

\begin{coro}
$\wt$ is geometrically maximal among all 3-crossing links in $T\times (0,1)$.
\end{coro}
\begin{proof}
From the preceding theorem, $vol(\wt)=4v_{oct}$. Since it contains two 3-crossings on the torus, its triple-crossing number is $c_3(\wt)= 2$. This implies that the \textit{triple volume density} of $\wt$ is
\begin{equation}
\mathcal{D}^3_{vol}(\wt)=\frac{vol(\wt)}{c_3(\wt)}=2v_{oct}
\end{equation}
This also realizes the MCCB bound for 3-crossings in general, where the region surrounding each 3-crossing can be decomposed into two octahedra. From the MCCB bound it follows that no link embedded in $T\times (0,1)$ can have a higher 3-crossing volume density.
\end{proof}

For finite links with 3-crossing planar projections in $S^3$, the MCCB bound of $2v_{oct}$ per crossing given by \cite{Adams2012} (and the equivalent decomposition above) certainly holds. However volume bound improvements can be made by collapsing the finite $U$ and $D$ vertices to the cusp, so equality is unattainable. In \cite{CKP} Champanerkar, Kofman, and Purcell were able to show that certain sequences of finite links that contain ever-increasing patches of the square weave also approach the infinite square weave in volume density. Their argument used lower bounds on volume attained by guts, which were derived from the essentiality of the checkerboard surfaces that came from the alternating projections that they considered. We expect analogous sequences of finite links containing increasing patches of the triple weave as in Figure \ref{patch} to similarly approach $\mathcal{D}^3_{vol}(\wt)=2v_{oct}$ in triple volume density, but the corresponding theory for links in triple-crossing projections is not yet developed enough to permit a similar argument.

\begin{figure}[htpb]
\centering
           \includegraphics[scale=.3]{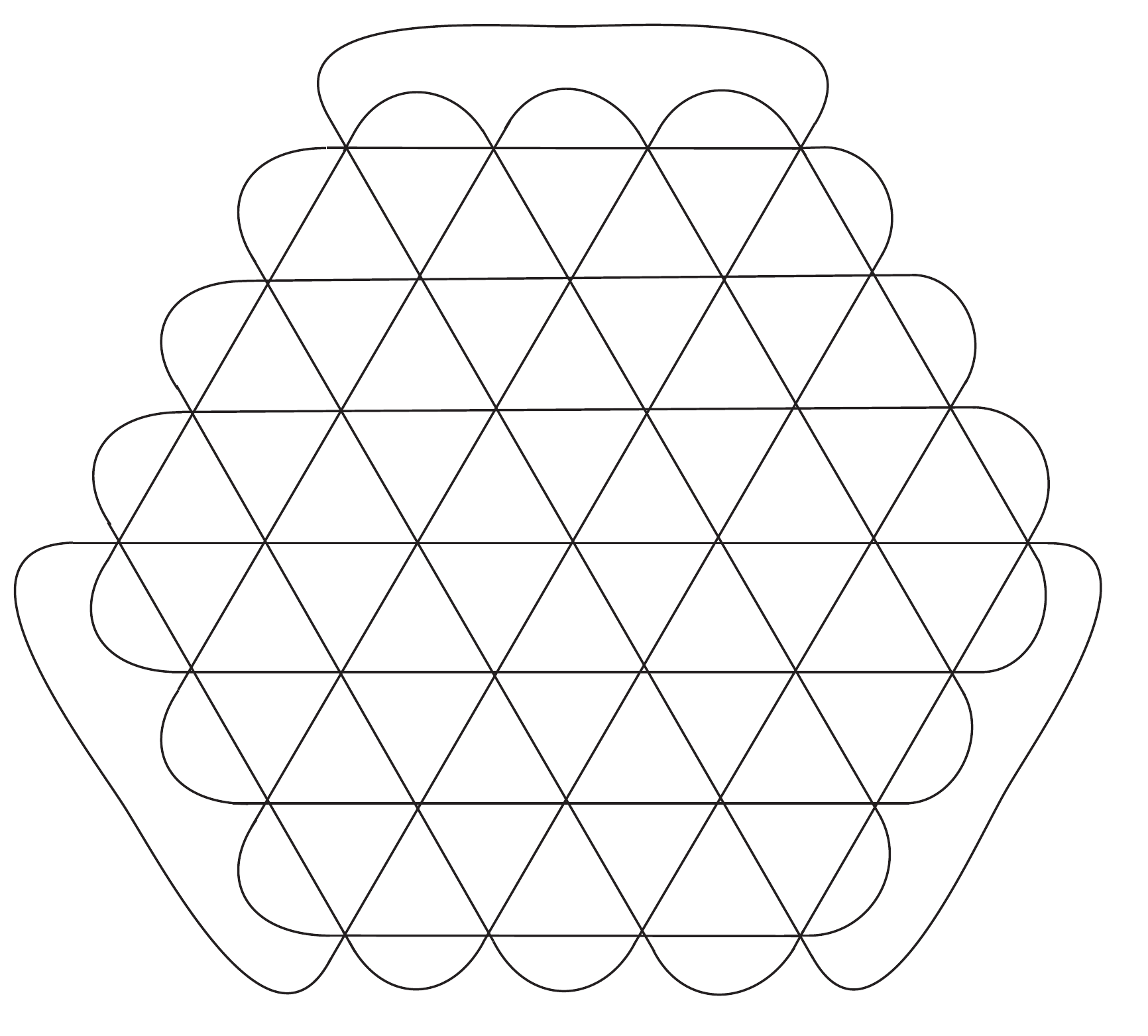}
\caption{Links containing ever larger patches of the triple weave should have volume density approaching $2v_{oct}$.}
  \label{patch}
\end{figure}

%give more credit here to triple crossings paper?

\begin{conj}
For links $L$ in a 3-crossing projection on the plane, the triple-crossing volume density bound
\begin{equation*}
\mathcal{D}^3_{vol}(L)< 2v_{oct}
\end{equation*}
is sharp, and is realized by a sequence of links as in Figure \ref{patch}.
\end{conj}

\bibliographystyle{plain}
\bibliography{GThesis}

\end{document}